\documentclass[reqno, 12 pt]{amsart}

\usepackage{latexsym}
\usepackage[english]{babel}
\usepackage[letterpaper,top=2cm,bottom=2cm,left=3cm,right=3cm,marginparwidth=1.75cm]{geometry}
\usepackage{amssymb}
\usepackage{amsfonts}
\usepackage{amsthm}
\usepackage{amsmath}
\usepackage{graphicx}
\usepackage{relsize}
\usepackage{hyperref}

\hypersetup{
    colorlinks = true,
    linkcolor  = blue,
    citecolor  = blue
}

\numberwithin{equation}{section}

\begin{document}
\title[The set of limits of integral sums of a multifunction]{The set of limits of Riemann integral sums of a multifunction and Banach space geometry}

\author[D.~Slobodianiuk]{Denys Slobodianiuk}

\address{School of Mathematics and Computer Science \\
 V.~N.~Karazin Kharkiv National University \\
Svobody square~4 \\
61022~Kharkiv \\ Ukraine
}

\email{slobodyanyukden@gmail.com}

\maketitle
\newtheorem{theorem}{Theorem}[section]
\newtheorem{lemma}[theorem]{Lemma}
\newtheorem{definition}[theorem]{Definition}
\newtheorem*{remark}{Remark}
\newtheorem*{corollary}{Corollary}

\begin{abstract}
Let $X$ be a Banach space and $F: [0, 1] \rightarrow 2^{X} \setminus \{ \varnothing \}$ be a bounded multifunction. We study properties of the set $I(F)$ of limits in Hausdorff distance of Riemann integral sums of $F$. The main results are:\\
    $(1)$ $I(F)$ is convex in the case of finite-dimensional $X$; \\
    $(2)$ $I(F) = I(\operatorname{conv} F)$ in B-convex spaces or for compact-valued multifunctions; \\
    $(3)$ $I(F)$ consists of convex sets whenever $X$ is B-convex; \\
    $(4)$ $I(F)$ is star-shaped (thus non-empty) for compact-valued multifunctions in separable spaces.\\
    $(5)$ For each infinite-dimensional Banach space there is a bounded multifunction with empty $I(F)$.
\end{abstract}

\section{Introduction}
Riemann integral, being historically the first rigorously defined type of integral, is at the same time the easiest to generalize for the case of vector-valued functions. Specifically, let $X$ be a Banach space,  $f : [0, 1] \rightarrow X$ be a function, $(\Gamma, T)$, $\Gamma = \{\Delta_i\}_{i = 1}^{n}$, $T = \{\xi_i\}_{i = 1}^{n}$ with  $\xi_i \in [a_i, b_i] = \Delta_i$, be a tagged partition of $[0,1]$ into segments, and  $d(\Gamma) := \max_{i} |\Delta_i|$ be the \emph{diameter of the partition}. Exactly like for real-valued functions, one defines the \emph{Riemann integral sum}
\[ S(f, \Gamma, T) := \sum_{i = 1}^{n} (b_i - a_i) f(\xi_i) = \sum_{i = 1}^{n} |\Delta_i| f(t_i) \]
where $|\Delta_i|$ denotes the length of the segment. The function $f$ is called \emph{Riemann integrable} with 
\[\int_0^1 f(t) dt = x \in X\]
if $\lim\limits_{d(\Gamma) \to 0}  S(f, \Gamma, T) = x$.

When above limit does not exist, it makes sense to consider the set $I(f)$ of all limits of integral sums, i.e. the set of all elements of $X$ that can be represented as the limit of a sequence $\{S(f, \Gamma_k, T_k)\}_{k = 1}^{\infty}$ with $d(\Gamma_k) \xrightarrow[k \rightarrow \infty]{} 0$.

It is a standard exercise for students that for bounded real-valued functions $I(f)$ is a segment connecting the lower and upper Darboux integrals, so it is a convex set. The convexity of $I(f)$ for bounded functions with values in finite-dimensional spaces is a non-trivial fact. It was demonstrated by Hartman \cite{Hartman} in 1947. Shortly afterwards, in 1954, Halperin and Miller \cite{Halperin} extended the convexity result to infinite-dimensional Hilbert spaces. Later, Nakamura  and Amemiya \cite{NakamuraAmemiya} and independently Hartman \cite{Hartman2} demonstrated the convexity of $I(f)$ for bounded functions with values in some non-Hilbertian Banach spaces. In addition, Nakamura  and Amemiya \cite{NakamuraAmemiya} presented an example of a bounded function $f$ with values in the non-separable space $\ell_1[0, 1]$, for which  $I(f)$ is not convex. Such an example in separable $\ell_1$ was constructed independently by M.Kadets and V.Kadets in \cite{Kadets1}.The intriguing question, for which Banach spaces the convexity of   $I(f)$ is true, was attacked in \cite{Kadets2, Kadets3}, see Appendix in the monograph \cite{Kadets} for a detailed exposition of the theory. Nevertheless, the question of complete characterization of such spaces remains open, and even for the classical space $c_0$ it is unknown whether all $c_0$-valued bounded functions enjoy the  $I(f)$ convexity.

 The aim of this paper is to find out to what extend the known results about $I(f)$ generalize to multifunctions, that is to functions whose values are non-empty subsets of a Banach space. In the construction that we explain below, the set $I(F)$ of limits of Riemann integral sums of a multifunction $F$ will be a collection of subsets. So the question of convexity will arise in two different senses. The first question will be the convexity of each element of $I(F)$. The second question will be the convexity of the whole collection $I(F)$, that is the property that for each $A,B \in I(F)$ their convex combinations $\lambda A + (1-\lambda)B$ belong to $I(F)$, where sum of sets and multiplication of a set by a scalar are defined in the standard way: \\
 for subsets $A, B$ of a Banach space $X$ and $\lambda \in \mathbb{R}$
\[A + B = \{a + b: a \in A, b \in B\}  \; , \;\;\; \lambda A = \{\lambda a: a \in A\}.\]

Let $X$ be a Banach space and $F: [0,1] \rightarrow 2^{X} \setminus \{\varnothing\}$ be a multifunction. 
The Riemann integral sums of $F$ are
\[ S(F, \Gamma, T) = \sum_{i = 1}^{n} (b_i - a_i) F(\xi_i) = \sum_{i = 1}^{n} |\Delta_i| F(t_i). \]
Let $M(F) = \sup\{\| F(t) \|: t \in [0, 1] \}$, where $\| F(t) \| = \sup \{\| a\|: a \in F(t)\}$. \\
We call a multifunction $F$ \emph{bounded} if $M(F)$ is finite, \emph{compact-valued} if $F(t)$ is compact for all $t \in [0, 1]$, and \emph{convex-valued} if $F(t)$ is convex for all $t \in [0, 1]$.

\begin{definition}
    Hausdorff distance between non-empty sets in a metric space is
    \[d_H(A, B) = \max\{\sup \{d(a, B): a \in A\}, \sup \{d(b, A): b \in B\}\}.\]
\end{definition}
\begin{remark}
    Since $d_H(A, \overline{A}) = 0$, the limit in Hausdorff distance is not uniquely defined. In order to eliminate the ambiguity, we will
    allow only closed limit sets. This allows us to consider the set of non-empty closed bounded sets as a metric space.
\end{remark}

\begin{definition}
    For bounded multifunction $F : [0, 1] \rightarrow 2^X \setminus \{\varnothing\}$ we call the set of all closed non-empty subsets of X 
    that can be represented as the limit of a sequence $\{S(F, \Gamma_k, T_k)\}_{k=1}^{\infty}$ with  $d(\Gamma_k) \xrightarrow[k \rightarrow \infty]{} 0$ the set of limits of Riemann integral sums of the multifunction and denote it $I(F)$.
\end{definition}

\begin{remark}
    It is trivial that $I(F)$ is closed. In addition, if $F$ is bounded, it is bounded by $M(F)$.
\end{remark}

\begin{definition}
    We denote the convex hull of a set as $\operatorname{conv}$ and define a multifunction \\
    $\operatorname{conv} F : [0, 1] \rightarrow 2^X \setminus \{\varnothing\}$ as $(\operatorname{conv} F)(t) = \operatorname{conv}(F(t))$ for each $t \in [0, 1]$.
\end{definition}

\begin{remark}
     The Minkowski addition and $\operatorname{conv}$ commute  \cite[theorem 1.1.2]{Schneider}, that is $\operatorname{conv} \sum_{k=1}^n A_k = \sum_{k = 1}^n \operatorname{conv} A_k$.
     This allows us to use $S(\operatorname{conv} F, \Gamma, T)$ and $\operatorname{conv} S(F, \Gamma, T)$ interchangeably.
\end{remark}

We finish this section by providing proofs for a few rather trivial technical lemmas. The reader familiar with the Hausdorff distance can skip the remainder of this section.

\begin{lemma} \label{convex limit} 
    Let $\{A_k\}_{k=1}^{\infty}$ be a sequence of non-empty convex sets that converge to a closed set $A$. Then $A$ is also convex. 
\end{lemma}
\begin{proof} 
   We need to show that $\forall x^{1}, x^{2} \in A$ and $\lambda \in [0,1] \;\; x = \lambda x^{1} + (1 - \lambda) x^{2} \in A$. Since $d(x^{i}, A_k) \rightarrow 0$,
   there exist sequences $\{x_{k}^{i} \in A_k\}_{k = 1}^{\infty}$ converging to $x^{i}$ accordingly. Convexity of the sets guarantees that $x_k = \lambda x_{k}^{1} + (1 - \lambda) x_{k}^{2} \in A_k$.
   The fact that $d(x_k, A) \rightarrow 0$ implies $\lim_{k\rightarrow \infty} x_k = x \in \overline{A} = A$.
\end{proof}

\begin{lemma} \label{conv distance}
    An inequality $d_H(\operatorname{conv} A, \operatorname{conv} B) \leq d_H(A, B)$ holds for any non-empty sets $A, B$.
\end{lemma}
\begin{proof}
    \[d_H(\operatorname{conv} A, \operatorname{conv} B) = \max\{\sup_{a \in \operatorname{conv} A} d(a, \operatorname{conv} B), \sup_{b \in \operatorname{conv} B} d(b, \operatorname{conv} A) \}. \]
    It is enough to demonstrate that 
    \begin{equation} \label{eq-1.1}
    \sup \{d(a, \operatorname{conv} B) : a \in \operatorname{conv} A\} \leq \sup\{ d(a, B) : a \in  A\},
    \end{equation}
    since the similar inequality for the second supremum can be obtained by swapping $A$ and $B$. \\
    For any $a \in \operatorname{conv} A$ consider its representation as a convex combination $a = \sum_{k = 1}^{n} \lambda_k a_k$, where $a_k \in A$.
    For each $a_k$ there is $b_k^{\varepsilon} \in B$ such that $\|a_k - b_k^{\varepsilon} \| \leq d(a_k, B) + \varepsilon \leq \sup\{ d(a, B) : a \in  A\} + \varepsilon$, thus
    \begin{eqnarray*}
    d(a, \operatorname{conv} B) &\leq& d( \sum_{k = 1}^{n} \lambda_k a_k,  \sum_{k = 1}^{n} \lambda_k b_k^{\varepsilon}) \leq \sum_{k = 1}^{n} \lambda_k \| a_k - b_k^{\varepsilon}\| \\
            &\leq& \sum_{k = 1}^{n} \lambda_k (\sup_{a \in A} d(a, B) + \varepsilon) = \sup_{a \in A} d(a, B) + \varepsilon,
    \end{eqnarray*}
    which, by arbitrariness of $\varepsilon$, gives us the desired inequality \eqref{eq-1.1}.
\end{proof}

\begin{lemma} \label{limit of convex}
    Let A be a convex set. Then for a sequence of non-empty sets $\{A_k\}_{k=1}^{\infty}$ convergence to A implies $\operatorname{conv} A_k \rightarrow A$.
\end{lemma}
\begin{proof}
    Convergence means that $d_H(A_k, A) \rightarrow 0$. The previous lemma applied to pairs $A_k, A$ gives that $d_H(\operatorname{conv} A_k, \operatorname{conv} A) \rightarrow 0$.
    It is left to notice that $\operatorname{conv} A = A$, therefore $\operatorname{conv} A_k \rightarrow A$.
\end{proof}

\begin{lemma} \label{limit of totally bounded}
    If sequence of non-empty and totally bounded sets $\{A_k\}_{k=1}^{\infty}$ converges to A, the limit is totally bounded,
\end{lemma}
\begin{proof}
    We need to demonstrate existence of an $\varepsilon$-net for an arbitrary $\varepsilon$. Fix $A_k$ satisfying $d_H(A_k, A) < \frac{\varepsilon}{2}$ and let $N$ be an $\frac{\varepsilon}{2}$-net of $A_k$. We calim that $N$ is an $\varepsilon$-net of A. \\
    Indeed, for any element $a \in A$ there exists $a_k \in A_k$ such that $\|a - a_k \| \leq d_H(A, A_k)$ and for any element $a_k \in A_k$ there is an element $n \in N$ such that $\|a_k - n\| < \frac{\varepsilon}{2}$. \\
    We conclude that claim holds, since $\|a - n\| \leq \|a - a_k\| + \|a_k - n\| < \frac{\varepsilon}{2} + \frac{\varepsilon}{2} = \varepsilon$.
\end{proof}

\section{Connection between $I(F)$ and $I(\operatorname{conv}F)$}
It turns out that for some classes of multifunctions and/or Banach spaces $I(F) = I(\operatorname{conv} F)$,
which allows one to only consider convex-valued multifunctions in further investigations of the properties of $I(F)$.
In additions, later in this section we demonstrate that relation between these two sets is connected to the convexity of their elements. \\
The first two subsection deal with statements similar to known results for the Riemann integral of a multifunction  \cite{Shevchenko},
because in the context of integrable functions the condition $I(F) = I(\operatorname{conv} F)$ is just another way of saying that the integrals are equal.

\subsection{Bounded multifunctions}

\begin{definition}
    A Banach space $X$ is said to have infratype $p$ with constant $C$ if the inequality 
    \[ \min_{a_i = \pm 1}{\left\| \sum_{k = 1}^{n} a_k x_k\right\|} \leq C\left(\sum_{k = 1}^{n} \|x_k \|^{p}\right)^{\frac{1}{p}} \]
    holds for every finite collection of elements $\{x_k\}_{k=1}^n \subset X$.
    A Banach space $X$ is said to have nontrivial infratype if $X$ has infratype $p > 1$ with some $C$.  Spaces with non-trivial infratype are also called B-convex.
\end{definition}

One of the reasons why this concept is important for us it the following result.

\begin{theorem}[{\cite[page 121]{Kadets}}] \label{theorem 1} 
    Let $X$ be a Banach space with nontrivial infratype and $f: [0, 1] \rightarrow X$ be a bounded function. Then $I(f)$ is a convex set. 
\end{theorem}

\begin{remark}[{\cite[lemma 2.2.1]{Kadets}}]
    Every finite-dimensional space has nontrivial infratype.
\end{remark}

\begin{remark}
    The original definition of B-convexity says that $X$ is B-convex if $\ell_1$ is not finitely representable in $X$. The equivalence of B-convexity and nontriviality of infratype is a deep result from local theory of Banach spaces \cite{Pisier}.
\end{remark}

\begin{lemma}[{\cite[lemma 3.4]{Shevchenko}}] \label{Shevchenko}
    Let $X$ be a Banach space with nontrivial infratype $p$ with constant $C$ and $F: [0, 1] \rightarrow 2^{X} \setminus \{\varnothing\}$ be a bounded multifunction. 
    Then the following inequality holds for all tagged partitions: 
    \[d_H(S(F, \Gamma, T), S(\operatorname{conv} F, \Gamma, T)) \leq C_1 M(F) d(\Gamma)^{\frac{p-1}{p}},\]
    where $C_1 = \frac{2C}{2^{1-\frac{1}{p}}-1}$.
\end{lemma}

\begin{theorem} \label{infratype equivalence}
    Let $X$ be a Banach space with nontrivial infratype $p$ and $F: [0, 1] \rightarrow 2^{X} \setminus \{\varnothing\}$ be a bounded multifunction. 
    Then $I(F) = I(\operatorname{conv} F)$.
\end{theorem}
\begin{proof}
    For an arbitrary $A \in I(\operatorname{conv} F)$ there exists a sequence of tagged partitions such that $d_H(S(\operatorname{conv} F, \Gamma_k, T_k), A) \rightarrow 0$ with $d(\Gamma_k) \xrightarrow[k \rightarrow \infty]{} 0$.
    \[d_H(S(F, \Gamma_k, T_k), A) \leq d_H(S(F, \Gamma_k, T_k), S(\operatorname{conv} F, \Gamma_k, T_k)) + d_H(S(\operatorname{conv} F, \Gamma_k, T_k), A)\] 
    Lemma \ref{Shevchenko} gives that $C_1 M(F) d(\Gamma_k)^{\frac{p-1}{p}} + d_H(S(\operatorname{conv} F, \Gamma_k, T_k), A)$
    is an infinitesimal squeezing upper bound for the sequence $\{d_H(S(F, \Gamma_k, T_k), A)\}_{k=1}^{\infty}$. Hence, $A \in I(F)$, i.e, $I(\operatorname{conv} F) \subset I(F)$. \\
    The same approach shows that $I(F) \subset I(\operatorname{conv} F)$: \\
    For an arbitrary $A \in I(F)$ there exists a sequence of tagged partitions such that \\
    $d_H(S( F, \Gamma_k, T_k), A) \rightarrow 0$ with $d(\Gamma_k)\xrightarrow[k \rightarrow \infty]{} 0$.
    \[d_H(S(\operatorname{conv} F, \Gamma_k, T_k), A) \leq d_H(S(F, \Gamma_k, T_k), S(\operatorname{conv} F, \Gamma_k, T_k)) + d_H(S(F, \Gamma_k, T_k), A)\] 
    Lemma \ref{Shevchenko} gives that $C_1 M(F) d(\Gamma_k)^{\frac{p-1}{p}} + d_H(S(F, \Gamma_k, T_k), A)$
    is an infinitesimal squeezing upper bound for the sequence  $\{d_H(S(\operatorname{conv} F, \Gamma_k, T_k), A)\}_{k=1}^{\infty}$. Hence, $A \in I(F)$, i.e, $I(F) \subset I(\operatorname{conv} F)$. Thus, $I(F) = I(\operatorname{conv} F)$.
\end{proof}

\begin{remark}
    In the previous theorem, in addition to the fact that $I(F) = I(\operatorname{conv} F)$, we demonstrated that for any sequence of tagged partitions with $d(\Gamma_k) \xrightarrow[k \rightarrow \infty]{} 0$
    the limits of corresponding Riemann sums for $F$ and $\operatorname{conv} F$ coincide, provided that it exists for  at least one of the multifunctions.
\end{remark}

\begin{remark}
In fact, for a Banach space $X$ the following conditions are equivalent \\
(i) for every bounded multifunction $I(F) = I(\operatorname{conv} F)$; \\
(ii) $X$ has nontrivial infratype.
\end{remark}
It is left to explain why for every Banach space without nontrivial infratype there exists a bounded multifunction $F$ so that $I(F) \ne I(\operatorname{conv} F)$.
Considering earlier remark pertaining to the connection between the infratype and B-convexity, existence is explained in  \cite[subsection 3.2]{Shevchenko}:
authors construct an example of a multifunction $F(t) = A$ such that $A \in I(\operatorname{conv} F)$ and $ A \notin I(F)$.
However, they do not claim that $A \notin I(F)$ as they deal with the Riemann integral rather than with a set of limits $I(F)$, i.e.,
it was enough to show that there exists a sequence of tagged partitions with $d(\Gamma_k) \xrightarrow[k \rightarrow \infty]{} 0$ such that $S(F, \Gamma_k, T_k) \not \rightarrow A$.
In fact, their proof for a fixed sequence works for any sequence of tagged partitions word for word after one explains how to properly choose $n$ ( \cite[formula 3.3]{Shevchenko}) for each individual partition.
One can verify that $n_k$ satisfying $2^{n_k-1} \geq |T_k|$ suffice.

\subsection{Compact-valued bounded multifunctions}
Due to the fact that investigation of the $I(f)$ for bounded function is {equivalent} to the investigation of $I(F)$ for singleton-valued bounded multifuctions,
a natural way to generalize is to consider only compact-valued bounded multifunctions.

\begin{definition}
    A Banach space $X$ is said to have the approximation property, if for every compact subset $K$ and $\varepsilon$ there exists a finite-rank operator
    $P: X \rightarrow X$ such that $\|Px - x\| < \varepsilon$ holds for all $x \in K$.
\end{definition}

\begin{lemma} \label{Compact with approx}
    Let $X$ be a Banach space that has the approximation property and $F: [0, 1] \rightarrow 2^{X} \setminus \{\varnothing\}$ be a compact-valued bounded multifunction. Then $I(F) = I(\operatorname{conv} F)$.
\end{lemma}
\begin{proof}
    It is enough to demonstrate inclusions $I(F) \subset I(\operatorname{conv} F)$ and $I(\operatorname{conv} F) \subset I(F)$. \\
    We use a similar approach for both inclusions. Specifically, we consider an arbitrary element (together with a sequence of tagged partitions)
    from the limit set of one function and demonstrate that it belongs to the limit set of the other function. More precisely, we show that it can be represented as the limit of Riemann sums of the same tagged partitions.
    Since the proofs of both inclusions admit similar approach, we label multifunctions $F$, $\operatorname{conv} F$ as $G$, $H$ in an arbitrary order.

    Fix $A \in I(G)$ and a corresponding sequence of tagged partitions with $d(\Gamma_k) \xrightarrow[k \rightarrow \infty]{} 0$.
    We start with the notion that $S(G, \Gamma_k, T_k) \rightarrow A$ and claim that $S(H, \Gamma_k, T_k) \rightarrow A$. We fix the sequence of tagged partitions and start denoting $S(*, \Gamma_k, T_k)$ as $S_k(*)$.

    Firstly, we notice that Lemma \ref{limit of totally bounded} implies that $A$ is a compact.

    We proceed by fixing an arbitrary $\varepsilon$, for which the approximation property guaranties existence of a finite-rank operator
    $P_A^{\varepsilon}$ such that $\|P_A^{\varepsilon}(x) - x\| < \varepsilon$ holds for all $x \in A$.
    For continuous operators $P_A^{\varepsilon}$ and $Q_A^{\varepsilon} = Id - P_A^{\varepsilon}$ original convergence implies that
    \[S_k(P_A^{\varepsilon} \circ G) \rightarrow P_A^{\varepsilon}(A) \text{ and } S_k(Q_A^{\varepsilon} \circ G) \rightarrow Q_A^{\varepsilon} (A). \]
    Since $P_A^{\varepsilon}$ is a finite-rank operator, $P_A^{\varepsilon} \circ F$ is a bounded multifunction in a finite-dimensional $\operatorname{Im} P_A^{\varepsilon}$.
    Therefore, Theorem \ref{infratype equivalence}, specifically the first remark, applies to $\operatorname{Im} P_A^{\varepsilon}$ and
    $P_A^{\varepsilon} \circ F : [0, 1] \rightarrow 2^{\operatorname{Im} P_A^{\varepsilon}} \setminus \{ \varnothing \}$,
    that is, $S_k(P_A^{\varepsilon} \circ F) \rightarrow P_A^{\varepsilon}(A) \iff
    S_k(\operatorname{conv} P_A^{\varepsilon} \circ F) = S_k(P_A^{\varepsilon} \circ \operatorname{conv} F) \rightarrow P_A^{\varepsilon}(A)$.
    Therefore, $S_k(P_A^{\varepsilon} \circ H) \rightarrow P_A^{\varepsilon}(A)$. We conclude that there exists $M_{\varepsilon}$ such that for all $k > M_{\varepsilon}$
    \[d_H(S_k(P_A^{\varepsilon} \circ H), P_A^{\varepsilon}(A)) < \varepsilon \text{ and } d_H(S_k(Q_A^{\varepsilon} \circ G), Q_A^{\varepsilon}(A)) < \varepsilon \]
    Now notice that $\|Q_A^{\varepsilon}(A)\| =  d_H(Q_A^{\varepsilon}(A), \{ 0 \})< \varepsilon$ by the choice of $P_A^{\varepsilon}$. Therefore, 
    \[d_H(S_k(Q_A^{\varepsilon} \circ F), \{ 0 \}) = d_H(S_k(\operatorname{conv} Q_A^{\varepsilon} \circ F), \{ 0 \}) = d_H(S_k(Q_A^{\varepsilon} \circ G), \{ 0 \}) \leq \] 
    \[ \leq d_H(S_k(Q_A^{\varepsilon} \circ G), Q_A^{\varepsilon}(A)) + d_H(Q_A^{\varepsilon}(A), \{ 0 \}) < 2 \varepsilon \]
    Next we combine estimations: there exists $M_\varepsilon$ such that for all $k > M_\varepsilon$ \\
    Case $G = F, H = \operatorname{conv} F$:
    \[d_H((S_k(H), A)) = d_H(S_k(\operatorname{conv} F), A) = d_H(S_k(\operatorname{conv} (P_A^{\varepsilon} + Q_A^{\varepsilon}) \circ F), \ (P_A^{\varepsilon} + Q_A^{\varepsilon})(A)) = \]
    \[= d_H(S_k(\operatorname{conv} P_A^{\varepsilon} \circ F) + S_k(\operatorname{conv} Q_A^{\varepsilon} \circ F), \ P_A^{\varepsilon}(A) + Q_A^{\varepsilon}(A)) \leq \] 
    \[ \leq d_H(S_k(\operatorname{conv} P_A^{\varepsilon} \circ F), \ P_A^{\varepsilon}(A)) + d_H(S_k(\operatorname{conv} Q_A^{\varepsilon} \circ F), \ Q_A^{\varepsilon}(A)) \leq \]
    \[ \leq d_H(S_k(P_A^{\varepsilon} \circ H), \ P_A^{\varepsilon}(A)) + d_H(\operatorname{conv} Q_A^{\varepsilon}(A) \circ F, \ \{ 0 \}) + d_H(\{ 0 \}, \ Q_A^{\varepsilon}(A)) < \varepsilon + 2 \varepsilon + \varepsilon \]
    Case $G = \operatorname{conv} F, H = F$:
    \[d_H((S_k(H), A)) = d_H(S_k(F), A) = d_H(S_k((P_A^{\varepsilon} + Q_A^{\varepsilon}) \circ F), \ (P_A^{\varepsilon} + Q_A^{\varepsilon})(A)) \leq \]
    \[ \leq d_H(S_k(P_A^{\varepsilon} \circ F), \ P_A^{\varepsilon}(A)) + d_H(S_k(Q_A^{\varepsilon} \circ F), \ Q_A^{\varepsilon}(A)) \leq \]
    \[ \leq d_H(S_k(P_A^{\varepsilon} \circ H), \ P_A^{\varepsilon}(A)) + d_H(S_k(Q_A^{\varepsilon} \circ F), \ \{ 0 \}) + d_H(\{ 0 \}, \ Q_A^{\varepsilon}(A)) < \varepsilon + 2 \varepsilon + \varepsilon \]
    
    Since we can get this conclusion regardless of the choice of $\varepsilon$, we proved the claim, i.e., $S_k(H) \rightarrow A$.
    Thus, we demonstrated that $I(G) \subset I(H)$, that is, $I(F) \subset I(\operatorname{conv} F)$ and $I(\operatorname{conv} F) \subset I(F)$. 
\end{proof}

\begin{lemma}[{\cite[proof of the theorem 4.3]{Shevchenko}}] \label{Isometric approx embedding} 
      Any Banach space can be isometrically embedded in a Banach space with an approximation property.
\end{lemma}

\begin{theorem} \label{compact}
    Let $X$ be a Banach and $F: [0, 1] \rightarrow 2^{X} \setminus \{\varnothing\}$ be a compact-valued bounded multifunction. Then $I(\operatorname{conv} F) = I(F)$.
\end{theorem}
\begin{proof}
    Immediately follows from Lemma \ref{Isometric approx embedding} and Lemma \ref{Compact with approx}.
\end{proof}

\subsection{Convexity of limits}
The question about the convexity of an integral for multifunctions is investigated for various integrals. For instance, it is discussed in  \cite[Chapters 8.7]{Audin} for measurable set-valued maps.
The Riemann integral is known to be convex for integrable multifunctions  \cite[Theorem 2.4, Proposition 2.2]{Shevchenko}. It is reasonable to investigate whether this holds for elements of $I(F)$.\\
In the context of convexity of limits, it is crucial that we consider only closed limits.

\begin{lemma} \label{convex function}
    Let $X$ be a Banach space and $F : [0, 1] \rightarrow 2^X \setminus \{\varnothing\}$ be a convex-valued multifunction. Then all elements of $I(F)$ are convex.
\end{lemma}
\begin{proof} 
    Consider any $A \in I(F)$. There exists a sequence of tagged partitions so that $S(F,\Gamma_k, T_k) \rightarrow A$ with $d(\Gamma_k) \xrightarrow[k \rightarrow \infty]{} 0$.
    We notice that $S(F,\Gamma_k, T_k)$ are convex as a sum of convex sets. \\
    Lemma \ref{convex limit} applied to the $\{S(F,\Gamma_k, T_k)\}_{k=1}^{\infty}$ guaranties convexity of $A$.
\end{proof}

\begin{theorem} \label{theorem 4}
    Let X be a Banach space and $F: [0, 1] \rightarrow 2^{X}\setminus{ \{ \varnothing \} }$ be a bounded multifunction. The following conditions are equivalent: \\
    (i) $I(F) \subset I(conv F)$;\\
    (ii) all elements of $I(F)$ are convex.
\end{theorem}
\begin{proof}
    $(i) \Rightarrow (ii)$. Because of Lemma \ref{convex function} we have that $I(\operatorname{conv} F)$ consists of convex elements. Therefore all elements of its subset $I(F)$ must be convex. \\
    $(ii) \Rightarrow (i)$. 
    For any $A \in I(F)$ there exists a sequence of tagged partitions, for which $A_k = S(F,\Gamma_k, T_k) \rightarrow A$ with $d(\Gamma_k) \rightarrow 0$ as $k \rightarrow 0$.
    If we apply Lemma \ref{limit of convex} to $\{A_k\}_{k=1}^{\infty}$ and convex $A$, we immediately get that
    $\operatorname{conv} A_k = S(\operatorname{conv} F, \Gamma_k, T_k) \rightarrow A$ with $d(\Gamma_k) \rightarrow 0$ as $k \rightarrow 0$, i.e., $A \in I(\operatorname{conv} F)$.
\end{proof}

In Section 2.1, we discussed an example satisfying $I(\operatorname{conv} F) \not\subset I(F)$.
The next natural question to ask is whether there exists a multifunction such that $I(F) \not\subset I(\operatorname{conv} F)$, or equivalently, such that $I(F)$ contains a non-convex element.
The following example, suggested by Vladimir Kadets and published here with his kind permission, provides us with an instance of such a bounded multifunction. \\
Consider Banach space $X = L_1[0, 1]$ and multifunction $F$ defined in a following way: 
\[F(x) =
    \begin{cases}
        p\; E[\frac{2n-2}{2p}, \frac{2n}{2p}] & \text{if } x = \frac{2n - 1}{2p} \text{ for } p \in \mathbb{P}, n \in \{1, .., p\}\\
        \{0\} &  \text{otherwise} 
    \end{cases}, \]
where $E[a, b] \subset L_1[0, 1]$ is the set of all characteristic functions of measurable subsets of $[a, b]$ and $\mathbb{P}$ stays for the set of all prime numbers. \\
The function is bounded, since 
$$
\left\|p E[\frac{2n-2}{2p}, \frac{2n}{2p}] \right\| = \sup_A \int_A p \;d\lambda = \sup_A p \lambda(A) \leq p \lambda([\frac{2n-2}{2p}, \frac{2n}{2p}]) = 1.
$$
We claim that $E[0,1] \in I(F)$, that is $I(F)$ contains a set that is not convex. Indeed, for a sequence of tagged partitions
$\Gamma_p = \{\Delta_i^p = [\frac{2i-2}{2p}, \frac{2i}{2p}]\}_{i=1}^{p}$, $T_p = \{\xi_i^p = \frac{2i - 1}{2p}\}_{i=1}^{p}$ for prime $p$:
\[S(F,\Gamma_p,T_p) = \sum_{i=1}^{p} \|\Delta_i^p\| \xi_i^p = \sum_{i=1}^{p} \frac{1}{p} p E[\frac{2i-2}{2p}, \frac{2i}{2p}] = \sum_{i=1}^{p} E[\frac{2i-2}{2p}, \frac{2i}{2p}] = E[0, 1] \]
Hence, $\{S(f, \Gamma_p, T_p)\}_{p \in \mathbb{P}}$ converges to a closed set $E[0, 1]$ with $d(\Gamma_p) \xrightarrow[p \rightarrow \infty]{} 0$.

\section{Convexity of $I(F)$}
\begin{definition}[{\cite[page 131]{Kadets}}]
    We say that Banach space $X$ belongs to the convex class ($X \in \operatorname{CONV}$) if for any bounded function $f : [0, 1] \rightarrow X$ the set $I(f)$ is convex.
\end{definition}

\begin{definition}
    We say that Banach space $X$ belongs to the multiconvex class ($X \in \operatorname{MULTICONV}$) if for any bounded multifunction
    $F : [0, 1] \rightarrow 2^{X} \setminus \{ \varnothing \}$ the set $I(F)$ is convex.
\end{definition}

We notice that $\operatorname{MULTICONV} \subset \operatorname{CONV}$.
It follows from the fact that for each function $f$ we can consider singleton-valued $F(t) = \{f(t)\}$, for which $I(F) = \{ \{a\} | a \in I(f) \}$.
This observation provides us with examples of spaces that do not belong to the multiconvex class. For instance, as was mentioned in the introduction, $\ell_1 \notin \operatorname{CONV}$ and therefore not in $\operatorname{MULTICONV}$.

In terms of new definitions, in this section we demonstrate that finite-dimensional spaces belong to the multiconvex class.

\begin{lemma} \label{totally bounded}
    Let $X$ be a Banach space and $f: [0, 1] \rightarrow X$ be a function with totally bounded image.  Then $I(f)$ is convex.
\end{lemma}
\begin{proof}
    It is enough to show that for any $x^1, x^2 \in I(f), \; \lambda \in [0,1]$ an element $\lambda x^1 + (1 - \lambda) x^2$ belongs to $I(f)$. \\
    Consider their corresponding tagged partitions $S(f, \Gamma_k^{i}, T_k^{i}) \rightarrow x^i$ with $d(\Gamma_k) \xrightarrow[k \rightarrow \infty]{} 0$. \\
    For an arbitrary fixed $\varepsilon > 0$ we introduce finite-dimensional space $$X_{\varepsilon} = \operatorname{Lin} \{\text{a finite } \varepsilon\text{-net of } f([0, 1]) \}$$
    and corresponding function $f_{\varepsilon} : [0, 1] \rightarrow X_{\varepsilon} \subset X$ defined in a way that $\| f_{\varepsilon}(t) - f(t) \| < \varepsilon$. \\ 
    The sequences $\{S(f_\varepsilon, \Gamma_k^{i}, T_k^{i})\}_{k=1}^{\infty} \subset (M(f) + \varepsilon) \overline{B_{X_\varepsilon}}$ contain convergent subsequences
    $S(f_{\varepsilon}, \Gamma_{n_k^i}^{i}, T_{n_k^i}^{i}) \rightarrow x_{\varepsilon}^{i}$ due to compactness of $(M(f) + \varepsilon) \overline{B_{X_\varepsilon}}$, i.e.,
    $x_{\varepsilon}^{1}, x_{\varepsilon}^{2} \in I(f_\varepsilon)$.
    Applied to $f_\varepsilon$ and finite-dimensional Banach space $X_\varepsilon$, Theorem \ref{theorem 1} guaranties convexity of $I(f_\varepsilon)$,
    more specifically, $\lambda x_\varepsilon^1 + (1 - \lambda) x_\varepsilon^2 \in I(f_\varepsilon)$.\\
    Hence there exists $S(f_{\varepsilon}, \Gamma_{\varepsilon}, T_{\varepsilon})$ with $d(\Gamma_\varepsilon) < \varepsilon$ satisfying
    \[ \| S(f_\varepsilon, \Gamma_\varepsilon, T_\varepsilon) - \left(\lambda x_\varepsilon^1 + (1 - \lambda) x_\varepsilon^2 \right) \| < \varepsilon\]
    Since $\| f_{\varepsilon}(t) - f(t) \| < \varepsilon$ and consequently $\|x_{\varepsilon}^{i} - x^{i} \| \leq \varepsilon$, we make following estimations
    \[ \left\| S(f_{\varepsilon}, \Gamma_\varepsilon, T_\varepsilon) - \left(\lambda x^1 + (1 - \lambda) x^2 \right) \right\| < 2\varepsilon \Rightarrow
    \left\| S(f, \Gamma_\varepsilon, T_\varepsilon) - \left(\lambda x^1 + (1 - \lambda) x^2 \right) \right\| < 3 \varepsilon\]
    Described construction for infinitesimal sequence $\{\varepsilon_k \}_{k=1}^{\infty}$ results in existence of a sequence of tagged partitions such that
    $S(f, \Gamma_{\varepsilon_k} T_{\varepsilon_k}) \rightarrow \lambda x^1 + (1 - \lambda) x^2$ with $d(\Gamma_{\varepsilon_k}) \xrightarrow[k \rightarrow \infty]{} 0$. \\
    Thus, $\lambda x^1 + (1 - \lambda) x^2 \in I(f).$
\end{proof}

In the rest of the paper we use a Radstrom embedding  \cite{Radstr}, namely a particular instance that maps convex closed bounded sets into points of $l_\infty(B_{X^*})$.
This transformation allows us to look at multifunction as a function in another space. \\
Let $bc(X)$ be a set of all closed convex bounded non-empty subsets of $X$. 

\begin{definition} \label{def-Randstr}
    For a Banach space $X$ we define embedding $\varphi : bc(X) \rightarrow l_{\infty}(B_{X^{*}})$
    \[\varphi(A)(x^{*}) = \sup x^{*}(A) = \sup_{a \in A} x^{*}(a)\]
\end{definition}

It has following properties  \cite[theorems II.18, II.19]{cas-val}: \\
\indent (i) $\varphi(A) + \varphi(B) = \varphi(A + B)$; \\
\indent (ii) $\varphi(\lambda A) = \lambda \varphi(A) \;\;\; \forall \lambda \geq 0$; \\
\indent (iii) $\| \varphi(A) - \varphi(B) \| = d_H(A, B)$. \\

\begin{lemma} \label{embedding lemma}
    Let $X$ be a Banach space, $F: [0, 1] \rightarrow 2^{X} \setminus \{ \varnothing \}$ be a convex-valued bounded multifunction. Then \\
    (a) $S(F, \Gamma_k, T_k) \rightarrow L \iff S(\varphi(F), \Gamma_k, T_k) \rightarrow \varphi(L)$ for any convex set $L$; \\
    (b) $S(F, \Gamma_k, T_k) \rightarrow L \Rightarrow S(\varphi(F), \Gamma_k, T_k) \rightarrow \varphi(L)$ for any set $L$.
\end{lemma}
\begin{proof}
    For a convex-valued multifunctions any Riemann sum is a convex set, therefore (iii) implies that \[d_H(S(F, \Gamma_k, T_k), L) = \|\varphi(S(F, \Gamma_k, T_k)) - \varphi(L)\|\]
    for any convex set $L$. Next we notice that $\varphi(S(F, \Gamma_k, T_k)) = S(\varphi(F), \Gamma_k, T_k)$ because of (i) and (ii). This proves (a). The (b) immediately follows from (a) and the fact that, by Lemma \ref{convex limit}, the limit $L$ must be convex as a limit of convex sets $S(F, \Gamma_k, T_k)$.
\end{proof}

\begin{theorem}\label{totally bounded F}
Let $X$ be a Banach space with nontrivial infratype and $F : [0, 1] \rightarrow 2^X \setminus \{\varnothing\}$ be a multifunction with totally bounded image. Then $I(F)$ is convex. 
\end{theorem}
\begin{proof}
    Banach space X has nontrivial infratype; therefore, by Theorem \ref{infratype equivalence}, $I(F) = I(\operatorname{conv} F)$.
    We proceed to prove the convexity of $I(F)$ by demonstrating that $I(\operatorname{conv} F)$, consisting of convex sets, is itself convex.
    That is, for any $A_1, A_2 \in I(\operatorname{conv} F)$ and $\lambda \in [0, 1]$, the set $\lambda A_1 + (1 - \lambda) A_2$ is in $I(\operatorname{conv} F)$. \\
    Consider a corresponding tagged partitions $S(\operatorname{conv} F, \Gamma_k^{i}, T_k^{i}) \rightarrow L_i$ with $d(\Gamma_k^{i}) \xrightarrow[k \rightarrow \infty]{} 0$. \\
    By Lemma \ref{embedding lemma}(b), we know that $S(\varphi(\operatorname{conv} F), \Gamma_k^{i}, T_k^{i}) \rightarrow \varphi(A_i)$, that is $\varphi(A_1), \varphi(A_2) \in I(\varphi(\operatorname{conv} F))$.
    The function $f = \varphi(\operatorname{conv} F) : [0, 1] \rightarrow l_{\infty}(B_{X^{*}})$ satisfies the conditions of Lemma \ref{totally bounded}.
    Indeed, the image $f([0, 1]) = \varphi (\operatorname{conv} F)([0, 1])$ is totally bounded because $F([0, 1])$ is totally bounded, and $\varphi \circ \operatorname{conv}$ does not increase the distance, i.e.,
    $| \varphi(\operatorname{conv} A) - \varphi(\operatorname{conv} B) | = d_H(\operatorname{conv} A, \operatorname{conv} B) {\leq} d_H(A, B)$. \\
    Now we can use the fact that $I(f)$ is convex, namely $\lambda \varphi(A_1) + (1 - \lambda) \varphi(A_2) = \varphi(\lambda A_1 + (1 - \lambda) A_2) \in I(f)$. Thus, there exists a sequence of tagged partitions satisfying
    $S(\varphi(\operatorname{conv} F), \Gamma_k, T_k) \rightarrow \varphi(\lambda A_1 + (1 - \lambda) A_2)$ with $d(\Gamma_k) \rightarrow 0$ as $k \rightarrow 0$.
    For a convex set $\lambda A_1 + (1 - \lambda) A_2$, Lemma \ref{embedding lemma}(a) gives us that the sequence of tagged partitions satisfies
    $S(\operatorname{conv} F, \Gamma_k, T_k) \rightarrow \lambda A_1 + (1 - \lambda) A_2$ with $d(\Gamma_k) \rightarrow 0$ as $k \rightarrow 0$.
    Hence, we can conclude that $\lambda A_1 + (1 - \lambda) A_2 \in I(\operatorname{conv} F)$.
\end{proof}

\begin{theorem}\label{theorem 7}
    Let X be a finite-dimensional Banach space and $F: [0, 1] \rightarrow 2^{X} \setminus \{ \varnothing \}$ is bounded. Then $I(F)$ is convex.
\end{theorem}
\begin{proof}
    We claim that conditions of Theorem \ref{totally bounded F} are met. Indeed, finite-dimensional spaces have nontrivial infratype,
    and the boundedness of the multifunction implies that $F([0, 1]) \subset M(F) \overline{B}_{X}$ is totally bounded in a finite-dimensional space.
\end{proof}

\section{Existence of a limit point}
It is important to mention that properties discussed in previous sections do not imply that $I(F)$ is not empty, because empty set is convex and consists of convex elements.

\begin{definition}
    Let $X$ be a Banach space. We call set $S \subset X$ star-shaped if there exists $x_0 \in S$ such that for all $ x \in S$ and $ \lambda \in [0, 1]$ element $\lambda x_0 + (1 - \lambda) x$ also belongs to $S$.
\end{definition}

\begin{remark}
    The definition  implies that a star-shaped set $S$ is not empty because $x_0 \in S$.
\end{remark}

\begin{definition}
    Let $M$ be a Banach space. We denote $K(M)$ the metric space of all convex compact non-empty subsets of $M$ endowed with the Hausdorff distance. 
\end{definition}

\begin{remark} \label{K(M)}
    The space $K(M)$ inherits separability and completeness of $M$  (see \cite[theorem II.14]{cas-val} and \cite[pages 259-260]{KM}).
 \end{remark}

\begin{theorem} \label{star theorem}
    Let $X$ be a separable Banach space, $F : [0, 1] \rightarrow 2^X \setminus \{\varnothing\}$ be a compact-valued bounded multifunction. Then $I(F)$ is star-shaped.
\end{theorem}
\begin{proof}
    The idea of the proof is to reduce the problem to the following known result by V.Kadets \cite{Kadets4}  (see \cite[page 131]{Kadets} for the proof in English): let $Y$ be a separable Banach space and $f: [0, 1] \rightarrow Y$ be a bounded function. Then $I(f)$ is star-shaped. \\
    In order to start, remark that the conditions of Theorem \ref{compact} are met for $F$. This allows us to consider only convex-valued functions.
    We notice that the image of $f(t) = \varphi(F(t))$ belongs to a convex cone $\varphi(K(X))$. Therefore, all of its convex combinations, most importantly all Riemann sums, also do.
    Moreover, $\varphi(K(X))$ is complete because $K(X)$ is complete and $\varphi|_{K(X)}$ preserves distance. Hence, $I(f) \subset \varphi(K(X))$. \\
    The last observation implies that elements of $I(f)$ can be represented as $\varphi(L)$ for some convex $L$. Under this condition,
    Lemma \ref{embedding lemma}(a) guarantees that $\varphi : I(F) \rightarrow I(f)$ is a bijection.\\
    Next, we introduce a Banach space $Y = \overline{\operatorname{Lin}} { \varphi(K(X))} \subset l_{\infty}(B_{X^{*}})$.
    Since $K(X)$ {is separable}, the distance preserving embedding $\varphi|_{K(X)}$ and closed linear hull preserve separability.
    Hence, $Y$ is a separable Banach space, and $f(t) = \varphi(F(t)) : [0, 1] \rightarrow Y$ is bounded by $M(F)$, thus, by the Kadets' theorem cited above, $I(f)$ is star-shaped. \\
    It is left to conclude that $I(F)$ is also star-shaped because $\varphi$ is a bijection for which (i) and (ii) of Definition \ref{def-Randstr} hold.
\end{proof}

\begin{corollary}
     Let $X$ be a finite-dimensional Banach space, $F : [0, 1] \rightarrow 2^X \setminus \{\varnothing\}$ be a bounded multifunction. Then $I(F)$ is not empty.
\end{corollary}

Our next step is to construct an example of a multifunction for which $I(F) = \varnothing$. In addition, the example highlights that some limitation on the class of multifunctions in Theorem \ref{star theorem} is required. 

\begin{definition}
    Banach space $X$ has biorthogonal system bounded by $C$ if it contains $\{x_i\}_{i = 1}^{\infty} \in X$ and $\{f_i\}_{i = 1}^{\infty} \in X^{*}$ satisfying $f_i(x_j) = \delta_i^j$ and $\|x_i \| \|f_i\| \leq C$.
\end{definition}

\begin{remark}
    A trivial example of space that has a biorthogonal bounded by $1$ system is any infinite-dimensional Hilbert space. In fact, every infinite-dimensional Banach space has a bounded biorthogonal system.
    One way to construct such systems is to consider a Schauder basic system $\{ x_i\}_{i = 1}^{\infty}$ and coordinate functionals $\{ f_i\}_{i = 1}^{\infty}$  \cite[1.a.2, 1.a.5, 1.b]{Lin-Tza-I}.
\end{remark}

\begin{theorem} \label{empty}
    Let X be a Banach space with a biorthogonal system $\{x_i\}_{i = 1}^{\infty}$, $\{f_i\}_{i = 1}^{\infty}$ bounded by $C$.
    Then there exists a bounded convex-valued multifunction with $I(F) = \varnothing$.
\end{theorem}
\begin{proof}
    Without loss of generality $\|x_i\|= C + 1$, $f_i \in B_{X^{*}}$, since we can consider a system $\{\lambda_i x_i\}_{i = 1}^{\infty}$, $\{\frac{f_i}{\lambda_i}\}_{i = 1}^{\infty}$.\\
    Let $S$ be a set of all finite unions of a countable base of the standard topology of $[0, 1]$. \\
    Since $S$ is countable, there exists a bijection $\pi : \mathbb{N} \rightarrow S$. We construct $F$ in the following way:
    \[F(t) = \operatorname{conv} \{x_n | t \in \pi(n), n \in \mathbb{N} \} \]
    This function is convex-valued and bounded by $C+1$. We claim that there does not exits any sequence of tagged partitions with
    $d(\Gamma_k) \xrightarrow[k \rightarrow \infty]{} 0$, such that ${S(F, \Gamma_k, T_k)}_{k = 1}^{\infty}$ is fundamental. \\
    We suppose that such sequence of tagged partitions exists and denote $S_k(*)= S(*, \Gamma_k, T_k)$.
    Consider an arbitrary $n$ and choose $m \geq n$ such that $d(\Gamma_m) < \frac{1}{2 |T_n|}$.\\
    Let $A_{m}^{n} = T_m \setminus T_n \subset [0, 1]$.
    For each point from $A_{m}^{n}$ there exists a neighborhood from the countable base which separates it from $T_n$. Let $U_{m}^{n}$ be their finite union. \\
    Consider $f = f_{\pi^{-1}(U_{m}^{n})} = f_{i}$ and recall biorthogonality property: \\
    for all $t \in T_n$ and $x \in F(t)$ we have $f(x) = 0$, since $U_{m}^{n} \cap T_n = \varnothing$; \\
    for all $t \in A_{m}^{m}$ and $x \in F(t)$ we have $f(x) \in [0, 1]$ and is equal to 1 for $x = x_{i} \in F(t)$.

    Now we can make an estimation for the distance between the sums utilizing $f$:
    \[d_H(S_n(F), S_m(F)) = \| \varphi(S_n(F)) - \varphi(S_m(F) \| = \sup_{g \in B_{x^{*}}} |\sup g(S_n(F)) - \sup g(S_m(F)) | \geq \]
    \[ \geq |\sup f(S_n(F)) - \sup f(S_m(F))| = |0 - \sup f(S_m(F))| = \]
    \[= \sum_{j \; : \; t_j \in A_{m}^{n}} |\Delta_j| \times 1 \geq \sum_{T_m} |\Delta_j| - \sum_{T_m \cap T_n} |\Delta_j| \geq 1 - |T_n| d(\Gamma_m) > \frac{1}{2}\]
    Contradiction with fundamentality.
\end{proof}

\section{Suggestions for further investigation}

\noindent 1. Does convex class coincides with the multiconvex class ? \\
2. How to describe the multiconvex class ? \\
3. How to describe spaces for which $I(F) \subset I(\operatorname{conv} F)$ ? \\

\section*{Acknowledgement}
The author is grateful to his scientific supervisor professor Vladimir Kadets for guidance and insights.

\end{document}